\documentclass[12pt,a4paper]{amsart}

\usepackage[active]{srcltx}

\usepackage{enumerate}
\usepackage{amsmath,amssymb}
\usepackage{multirow}

\usepackage[all]{xy}

\newcommand{\Z}{{\mathbb Z}}

\newcommand{\F}{{\mathbb F}}

\newcommand{\Sq}[1]{\operatorname{Sq}^{#1}} 

\renewcommand{\P}[1]{\mathcal{P}^{#1}}

\newcommand{\order}{\operatorname{Order}\nolimits}

\newcommand{\Hom}{\operatorname{Hom}\nolimits}

\newcommand{\Id}{\operatorname{Id}\nolimits}

\newcommand{\GL}{\operatorname{GL}\nolimits}

\newcommand{\A}{\ifmmode{\mathcal{A}}\else${\mathcal{A}}$\fi}
\newcommand{\K}{\ifmmode{\mathcal{K}}\else${\mathcal{K}}$\fi}
\newcommand{\KM}{\ifmmode{\mathcal{M}}\else%
${\mathcal{M}}$\fi}
\newcommand{\U}{\ifmmode{\mathcal{U}}\else${\mathcal{U}}$\fi}
\newcommand{\M}{\ifmmode{\mathcal{M}}\else${\mathcal{M}}$\fi}
\newcommand{\N}{\ifmmode{\mathcal{N}}\else${\mathcal{M}}$\fi}
\newcommand{\Ff}{\ifmmode{\mathcal{F}}\else${\mathcal{F}}$\fi}
\newcommand{\Ll}{\ifmmode{\mathcal{L}}\else${\mathcal{L}}$\fi}
\newcommand{\X}{\ifmmode{\mathcal{X}}\else${\mathcal{X}}$\fi}
\newcommand{\Y}{\ifmmode{\mathcal{Y}}\else${\mathcal{Y}}$\fi}

\newtheorem{theorem}{Theorem}[section]
\newtheorem{proposition}[theorem]{Proposition}
\newtheorem{corollary}[theorem]{Corollary}
\newtheorem{lemma}[theorem]{Lemma}

\theoremstyle{definition}
\newtheorem{definition}[theorem]{Definition}
\newtheorem{remark}[theorem]{Remark}

\theoremstyle{remark}

\theoremstyle{plain}

\newcommand{\definicio}{\stackrel{\text{def}}{=}}

\title[Cohomology, Massey products and the MIP]{Cohomological uniqueness, Massey products and the modular isomorphism problem for $2$-groups of maximal nilpotency class}

\author{Albert Ruiz}
\address[A.~Ruiz]{Departament de Matem{\`a}tiques,
Universitat Aut{\`o}noma de Barcelona, 08193 Cerdanyola del
Vall{\`e}s, Spain.}
\email[A.~Ruiz]{Albert.Ruiz@uab.cat}

\author{Antonio Viruel}
\address[A.~Viruel]{
Departamento de {\'A}lgebra, Geometr{\'\i}a y Topolog{\'\i}a,
Universidad de M{\'a}\-la\-ga, Apdo correos 59, 29080 M{\'a}laga,
Spain.}
\email[A.~Viruel]{viruel@agt.cie.uma.es}

\thanks{
\textbf{Key words:} 2000 Mathematics subject classification 55R35,
20D15.\\ \indent First author is partially supported by FEDER-MEC
grant MTM2010-20692.\\ \indent Second author is partially
supported by FEDER-MEC grant MTM2010-18089 and Junta de
Andaluc{\'\i}a grants FQM-213 and P07-FQM-2863.\\ \indent Both
authors are partially supported by Generalitat de Catalunya grant
2009SGR-1092.}

\begin{document}

\begin{abstract}
Let $G$ be a finite $2$-group of maximal nilpotency class, and let $BG$ be its classifying space.
We prove that iterated Massey products in $H^*(BG;\F_2)$ do characterize the homotopy type of $BG$ among $2$-complete spaces with
the same cohomological structure. As a consequence we get an alternative proof of the modular isomorphism problem for $2$-groups of maximal nilpotency class.
\end{abstract}

\maketitle


\section{Introduction}

Let $G$ be a finite $p$-group, and let $BG$ be its classifying
space. In this work we consider the cohomological uniqueness of
$BG$: choose some mod-$p$ cohomological invariants and consider
$\X$ a $p$-complete space such that $BG$ and $\X$ agree on these
cohomological invariants, does it mean that $\X$ is homotopy
equivalent to $BG$? As one may expect, the answer to this question
depends on the mod-$p$ cohomological invariants chosen.

In \cite{BL}, C.~Broto and R.~Levi initiated the study of the
cohomological uniqueness of $BG$ in terms of Steenrod operations
and Bockstein spectral sequences. In this setting, they proved the
cohomological uniqueness of the classifying space of every
dihedral $2$-groups \cite{BL}, and every quaternion group
\cite{BL2}. Unfortunately, the available techniques seem not to be
strong enough to decide the cohomological uniqueness of the
classifying space of semidihedral $2$-groups in terms of Steenrod
operations and Bockstein spectral sequences.

In order to give a unified approach to the cohomological uniqueness of 
$BG$, when $G$ is a finite $2$-group of maximal nilpotency class 
(i.e.\ dihedral, quaternion and semidihedral $2$-groups), we propose a 
different set of mod-$p$ cohomology invariants:
the algebra structure of $H^*(BG;\F_p)$ (not taking into account Steenrod operations), 
and iterated Massey products in this algebra (see Section \ref{sec:def-massey} for precise definitions). 
In this setting we prove:

\begin{theorem}\label{main}
Let $G$ be a finite $2$-group of maximal nilpotency class. Let
$\X$ be a $2$-complete topological space having the homotopy type of a $CW$-complex such that
$H^*(\X;\F_2) \cong H^*(BG;\F_2)$ as algebras with iterated Massey products. Then
$\X \simeq BG$.
\end{theorem}
\begin{proof}
We first describe $H^*(BG;\F_2)$ as algebra with iterated Massey products in Section \ref{sec:calculo-massey}. This structure is used to construct a homotopy equivalence $\X\to BG$ whenever $G$ is dihedral (Theorem \ref{theoremD}), quaternion (Theorem \ref{theoremQ}) or
semidihedral (Theorem \ref{theoremSD}).
\end{proof}

Besides of its own topological interest, the study of the
cohomological uniqueness of classifying spaces may have
implications in a long standing algebraic problem: the modular
isomorphism problem.

The modular isomorphism problem asks whether a $p$-group $G$ is determined by  $\F_p[G]$, its group algebra over the field of $p$-elements. That is, given another finite $p$-group $H$ such that $\F_p[G]\cong\F_p[H]$ as rings, does it mean that $G\cong H$ as groups? The general question still remains open, a positive answer to this question is known for a few families of $p$-groups (see  \cite[Introduction]{baginski-kolanov} for an up to date list of results).

Then the interplay between cohomological uniqueness of classifying
spaces and the modular isomorphism problem is clear: let $G$ be a
finite $p$-group such that $BG$ is cohomology unique when
considered a set of mod-$p$ cohomological invariants that can be
deduced from the group algebra $\F_p[G]$, then the modular
isomorphism problem has a positive answer for $G$.

So we obtain a positive answer to the modular isomorphism problem for finite $2$-groups of maximal nilpotency class (see \cite{Carlson}, and \cite{Baginski-maxclass2} for other approaches to this result):

\begin{corollary}
Let $G$ be a finite $2$-group of maximal nilpotency class, and let $H$ be a finite $2$-group such that $\F_2[G]\cong\F_2[H]$ as rings, then $G\cong H$ as groups.
\end{corollary}
\begin{proof}
Given any finite $p$-group $K$, the algebra structure with iterated Massey products of $H^*(BK;\F_p)$ can be obtained from the ring structure of $\F_p[G]$ by means of the Yoneda cocomplex (e.g. see \cite[Theorem 2.3 and Lemma 2.4]{Borge-JPAA}). Hence, if $\F_2[G]\cong\F_2[H]$ as rings, then $H^*(BH;\F_2) \cong H^*(BG;\F_2)$ as algebras with iterated Massey products. Then, according to Theorem \ref{main}, $BG\simeq BH$ (recall that the classifying space of a finite $p$-group is a $p$-complete space), and $G\cong H$ as groups.
\end{proof}

We finish this introduction recalling that iterated Massey products of degree one classes have been previously used by I.C.~Borge in \cite{Borge-JPAA} and \cite{Borge-Thesis} to provide a cohomological classification of finite $p$-groups. Unfortunately, her results cannot be applied in our topological framework: in Remark \ref{remark:higherneeded}, we give an example showing that  iterated Massey products of degree one classes cannot isolate
the homotopy type of the classifying space of quaternion groups.

\noindent \emph{Organization of the paper:} In Section 2, we briefly review the definition, main
properties and tools needed in the computation of iterated Massey products. Section 3 is devoted
to list and give some properties of maximal nilpotency class $2$-groups. In Section 4,
we compute iterated Massey products in the cohomology of these groups and in Section 5, we use
this structure to prove Theorem \ref{main} case by case. There are also two appendices where we
develop some explicit computations needed in Section 4:
Appendix A provides us with explicit representations of maximal nipotency class $2$-gropus in 
$\GL_m(\F_2)$ for minimal $m$ and
Appendix B is devoted to express the generators in the cohomology of quaternion and semidihedral groups
as cochain morphisms in the Yoneda cocomplex. 

\noindent \emph{Notation:} In general the following notations will be used in
the rest of the paper: group elements are denoted with lower case letters ($x$,$y$, \dots), while roman capital letters ($X$, $Y$, \dots) and calligraphic capital letters ($\X$ and $\Y$) are used to denote cohomology generators,
and topological spaces respectively.

Unless otherwise stated, cohomology means cohomology with trivial coefficients over the field $\F_2$, so
$H^*(\X)=H^*(\X;\F_2)$. This also implies that when we talk about $2$-complete spaces
we are considering the $2$-completion in the sense of A.K.~Bousfield and D.~Kan \cite{BK}.


\section{Iterated Massey products: definition and properties}\label{sec:def-massey}
In this section we quickly review the theory of iterated Massey products.
A more detailed description can be found in \cite{Kraines}, or \cite{May}.

Let $R$ be a ring with unit. Consider $C^*(G,R)$ the cochain algebra of the group $G$ with
coefficients in $R$, and $d$ the coboundary operator in $C^*(G,R)$.

\begin{definition}
Let $\{X_i\}_{1\leq i \leq n}$ be homogeneous elements in $H^{*}(G,R)$.
A \emph{defining system for the $n$-fold iterated Massey product
$\langle X_1, \dots , X_n\rangle$},
is an
upper triangular matrix
$$
M=\{m_{i,j} \mid 1\leq i\leq n+1, i<j\leq n+1, (i,j) \neq (1,n+1) \}
$$
with coefficients in $C^*(G,R)$ such that:
\begin{enumerate}[(i)]
 \item $m_{i,i+1}$ is a representative for $X_i$ and
 \item $dm_{i,j}=\displaystyle\sum_{k=i+1}^{j-1} m_{i,k} \cup m_{k,j}$ ($j\neq i+1)$.
\end{enumerate}
\end{definition}

\begin{definition}
Given $M$ a defining system for $\langle X_1, \dots, X_n\rangle$, the
\emph{value of the product relative to this defining system}, denoted
$\langle X_1, \dots, X_n\rangle_M$, is the element in $H^*(G;R)$ represented
by the cocycle:
$$
\sum_{k=2}^n m_{1,k}\cup m_{k,n+1} .
$$
The \emph{$n$-fold iterated Massey product $\langle X_1, \dots, X_n\rangle$} is defined as the set of elements
which can be written as $\langle X_1, \dots, X_n\rangle_M$ for some defining
system $M$. The \emph{indeterminacy in the iterated Massey product $\langle X_1, \dots, X_n\rangle$} is defined
as the set of elements $Z$ which can be expressed as $Z=Y_1-Y_2$ for $Y_1$ and $Y_2$ in $\langle X_1, \dots, X_n\rangle$.
\end{definition}
We now enumerate some of the properties which are used later:
\begin{enumerate}[(i)]
\item The iterated Massey product $\langle X_1, \dots, X_n\rangle$ is not defined for all
$X_1, \dots, X_n$ in $H^*(G;R)$. For example, $\langle X_1,X_2,X_3\rangle$
is defined if and only if $X_1\cup X_2=X_2\cup X_3=0$.
\item The degree of an element in $\langle X_1, \dots, X_n\rangle$ is $\sum \deg(X_i)-n+2$.
\item If $f\colon \X \to \Y$ is a continuous map of topological spaces, and $Y_1, \dots, Y_r$
are cohomology classes in $H^*(\Y;R)$ such that $\langle Y_1, \dots,$  $Y_r \rangle$ is defined,
then so is $\langle f^*(Y_1),\dots, f^*(Y_r)\rangle$ and
$$f^*(\langle Y_1, \dots, Y_r \rangle)\subset \langle f^*(Y_1),\dots,
f^*(Y_r)\rangle.$$
Moreover, if $f^*$ is an isomorphism, equality holds.
\end{enumerate}

The next result follows from May's proof of \cite[Theorem 1.5]{May}:
\begin{lemma}\label{lemmaMay}
Let $f\colon \X \to \Y$ be a continuous map such that $f^k \colon
H^k(\Y;R)
\to H^k(\X;R)$ is an isomorphism for $k \leq n$. Let $Y_1, \dots, Y_r$ be elements
in $H^*(\Y;R)$ such that $\sum \deg(Y_i) -r +2  \leq n$. Then:
\begin{enumerate}[\rm (a)]
\item  $\langle Y_1, \dots , Y_r\rangle$ is defined if and only if $\langle f^*(Y_1), \dots,
f^*(Y_r)\rangle$ is so.
\item $f^*(\langle Y_1, \dots , Y_r\rangle)=\langle f^*(Y_1), \dots, f^*(Y_r)\rangle$.
\end{enumerate}
\end{lemma}

Let $p$ be a fixed prime. In this work, we compare the mod-$p$ cohomology algebras (with iterated Massey products) of spaces. This is done by considering a special kind of morphisms:

\begin{definition}
Let $\varphi\colon H^*(\Y;\F_p) \to H^*(\X;\F_p)$ be a morphism (not necessarily induced by a continuous
map of topological spaces). We say that $\varphi$ is an \textit{$\M$-isomorphism} if
\begin{enumerate}
\item $\varphi$ is an $\F_p$-algebras isomorphism and
\item for all $r \geq 1$ and $Y_1$, \dots , $Y_r$ elements in $H^*(\Y;\F_p)$ such
that $\langle Y_1, \dots , Y_r \rangle$ is defined, then $\langle \varphi(Y_1),
\dots \varphi(Y_r) \rangle$ is defined and
$$
\varphi(\langle Y_1, \dots , Y_r \rangle)=
\langle \varphi(Y_1), \dots \varphi(Y_r) \rangle
$$
\end{enumerate}
\end{definition}

\begin{definition}
Let $\X$ and $\Y$ be topological spaces. We say that $\X$ and $\Y$ are
\textit{$\M$-comparable} if there exists an
$\M$-isomorphism $\varphi \colon H^*(\Y;\F_p) \to H^*(\X;\F_p)$.
We say that $\X$ \textit{is determined by its $\M$-cohomology} if for any
$p$-complete space $\Y$ having the homotopy type of a $CW$-complex and such
that it is $\M$-comparable to $\X$ we have that $\X \simeq \Y$.
\end{definition}

The easiest examples of spaces determined by its $\M$-cohomology are provided by classifying spaces of some particular families of $p$-groups:

\begin{proposition}\label{example:elementaryabelian}
Let $E$ be an elementary abelian $p$-group, then $BE$ is determined
by its $\M$-cohomology.
\end{proposition}
\begin{proof}
Let $E$ be a rank $r$ elementary abelian $p$-group.  First we prove the algebra structure of $H^*(BE;\F_p)$ with iterated Massey products determines the Steenrod operations and the Bockstein spectral sequence. We consider two cases:

\begin{itemize}
\item If $p=2$, then $H^*(BE;\F_2)=\F_2[X_1,\ldots, X_r]$ where every $X_i$ has degree $1$. 
Therefore, the unstability axiom forces $\Sq{1}X_i=X_i^2$, and this determines completely the Steenrod 
operations and the Bockstein spectral sequence.
\item If $p>2$, then $H^*(BE;\F_2)=\Lambda(U_1,\ldots,U_r)\otimes\F_2[V_1,\ldots, V_r]$ where every 
$U_i$ has degree $1$, and every $V_i$ has degree $2$. 
Again, the unstability axiom forces $\P{1}U_i=0$ and $\P{1}V_i=V_i^p$, and this determines completely 
the Steenrod operations. 
Moreover, according to \cite[Theorem 19]{Kraines}, the $p$-fold iterated Massey product 
$\langle U_i,U_i,\ldots,U_i\rangle$ equals $\beta_1(U_i)=V_i$,
where
$\beta_1$ is the primary Bockstein operator. So, in this case, Massey products completely determine the Bockstein spectral 
sequence.
\end{itemize}

Finally, $BE$ is determined by its cohomology with Steenrod operations and Bockstein spectral sequence \cite[Proposition 1.5]{BL}, thus the result follows.
\end{proof}

\subsection{Iterated Massey products of degree one elements}
In this subsection we recall the work of W.G. Dwyer \cite[Section 2]{Dwyer2}, that relates iterated Massey products of degree one elements in the cohomology of a group and representations of this group in the upper triangular
matrices.

Let $U(R,n)$ be the multiplicative group of all upper triangular $n\times n$ matrices over $R$ which agree with the identity matrix along the diagonal.
The subgroup $Z(R,n)$ of $U(R,n)$ consists of matrices which are identically zero except along the diagonal and at position $(1,n)$. We get that $Z(R,n)\cong R$ and it is contained in the center of
$U(R,n)$, so it gives rise to the central extension:
\begin{equation}\label{extcentral}
Z(R,n) \to U(R,n) \to \overline{U}(R,n) \definicio U(R,n)/Z(R,n) .
\end{equation}
Given a group homomorphism $\phi \colon G \to U(R,n)$, the image of $g\in G$ is a matrix with
coefficients $\phi_{i,j}(g) \in R$. Remark that the elements $\phi_{i,i+1}(g)$ satisfy the
equation:
$$
\phi_{i,i+1}(g_1g_2)=\phi_{i,i+1}(g_1)+\phi_{i,i+1}(g_2) .
$$
So $\phi_{i,i+1}$ are group morphisms from $G$ to $R$, and thus represent cohomology classes
in $H^1(G;R)$. These elements are called \emph{near diagonal elements} of $\phi$.

\begin{theorem}[\cite{Dwyer2}]\label{theoremdwyer}
Let $X_1, \dots, X_ n$ be elements in $H^1(G;R)$. There is a one-one correspondence
$M\leftrightarrow\phi_M$ between defining systems $M$ for $\langle X_1, \dots, X_n\rangle$
and group homomorphisms $\phi_M\colon G \to \overline{U}(R,n+1)$ which have $-X_1, \dots, -X_n$
as near-diagonal components. Moreover, given $\phi_M$, we can pull back the extension in Equation
\eqref{extcentral}, getting an extension of $G$ by $R$. Then $\langle X_1, \dots, X_n\rangle_M$
is exactly the characteristic class of this extension.
\end{theorem}

\subsection{Iterated Massey products in higher degrees}
In order to compute iterated Massey products of higher degree elements
we use the Yoneda cocomplex \cite[Section 2.5]{Borge-Thesis}:

Fix $G$ a $p$-group and $P_\bullet \to \F_p$ a $\F_p[G]$-free resolution with differentials
$\partial_i \colon P_{i+1} \to P_{i}$.
\begin{definition}
The \emph{Yoneda cocomplex} $\Hom^{(\bullet)}_{\F_p[G]}(P_\bullet,P_\bullet)$ is
defined as:
\begin{itemize}
\item In degree $i$ we have the $\F_p[G]$-module:
$$
\Hom^{i}_{\F_p[G]}(P_\bullet,P_\bullet) = \prod_{n\in\Z} \Hom_{\F_p[G]}(P_{n+i},P_{n}),
\quad i\geq 0
$$
and $0$ otherwise.
\item The differential of
$\phi^i=\{\phi^i_n \colon P_{n+i} \to P_n\}_{n\in\Z} \in \Hom^{i}_{\F_p[G]}(P_\bullet,P_\bullet)$,
is defined as:
$$
\delta^i(\phi_n^i)=\partial_{n-1}\phi_{n}^i-(-1)^i\phi_{n-1}^i\partial_{n+i-1} .
$$
\item The algebra structure in the cohomology of $\Hom^{(\bullet)}_{\F_p[G]}(P_\bullet,P_\bullet)$
is induced by the composition of elements as a cochain morphisms.
\end{itemize}
\end{definition}
The following result tells us that we can use this tool to work with the cohomology
of a group:
\begin{theorem}[\cite{Borge-Thesis}]
$H^i(\Hom_{\F_p[G]}^{(\bullet)}(P_\bullet,P_\bullet),\delta^\bullet) \cong H^i(G;\F_p)$ for all $i\geq 0$.
\end{theorem}


\section{Naximal nilpotency class 2-groups} 
The maximal nilpotency class finite $2$-groups are precisely the
dihedral, quaternion and
semidihedral groups.

Consider the following finite presentations of these three families of $2$-groups:
\begin{equation}\label{notationG}
\begin{array}{l}
D_{2^n}=\langle x , y \mid x^{2^{n-1}}=1, y^2=1 , yxy^{-1}=x^{-1}\rangle ,\\
Q_{2^n}=\langle x , z \mid x^{2^{n-1}}=1, z^2=x^{2^{n-2}} ,
zxz^{-1}=x^{-1}\rangle
\mbox{ and }\\
SD_{2^n}=\langle x , t \mid x^{2^{n-1}}=1, t^2=1 ,
txt^{-1}=x^{2^{n-2}-1}\rangle .
\end{array}
\end{equation}
And also the following cohomology rings with coefficients in $\F_2$:
\begin{equation}\label{notationHG}
\begin{array}{l}
H^*(BD_4)\cong \F_2[X,Y],  \\
H^*(BD_{2^n})\cong \F_2[X,Y,W]/(X^2+XY) \mbox{ for $n\geq 3$},\\
H^*(BQ_8)\cong \F_2[X,Y,V]/(X^2+XY+Y^2,X^2Y+XY^2), \\
H^*(BQ_{2^n})\cong \F_2[X,Y,V]/(X^2+XY,Y^3)  \mbox{ for $n\geq 4$ and}\\
H^*(BSD_{2^n}) \cong \F_2[X,Y,U,V]/(X^2\!+\!XY,XU,X^3,U^2+(X^2+Y^2)V) \\
 \quad\quad\quad\quad\quad \mbox{ for $n\geq 4$}.
\end{array}
\end{equation}
where $\deg(X)=\deg(Y)=1$, $\deg(W)=2$, $\deg(U)=3$ and $\deg(V)=4$ in
all the cases.

The results in this paper will not require that the morphisms between the cohomologies of
the spaces preserve the structure of 
algebra over the Steenrod
Algebra, but in some proofs we will use that such a structure exists.
More precisely we will
use that it is unstable and that it is known
that $\Sq{1}(W)=WY$ in $H^*(BD_{2^n})$ ($n\geq3$).

The following properties can be found in \cite{BL}. There is a tower of principal fibrations:
\begin{equation}\label{fibration}
\cdots \stackrel{\pi_{n}}{\to} BD_{2^n} \stackrel{\pi_{n-1}}{\to}
BD_{2^{n-1}} \stackrel{\pi_{n-2}}{\to}
\cdots \stackrel{\pi_{3}}{\to} BD_8 \stackrel{\pi_{2}}{\to} BD_4
\end{equation}
where $\pi_2$ is classified by the class $X^2+XY \in H^2(BD_4)$ and $\pi_n$ by
$W \in H^2(BD_{2^n})$ for $n\geq 3$.

The quaternion and semidihedral groups fit in the following central extensions:
\begin{equation}\label{extension}
\Z/2 \to Q_{2^n} \to D_{2^{n-1}} \mbox{ and } \Z/2 \to SD_{2^n} \to D_{2^{n-1}}
\end{equation}
classified by the following classes in $H^2(BD_{2^{n-1}})$: $X^2+XY+Y^2$ in the case $Q_8$, $W+Y^2$ in the
case $Q_{2^n}$ ($n\geq4$) and
$W+X^2$ in the case $SD_{2^n}$ ($n\geq4$).


\section{Iterated Massey products in the cohomology of maximal nilpotency class 2-groups}
\label{sec:calculo-massey}

This section is devoted to the computation of the iterated Massey products in
the cohomology
of dihedral, quaternion and semidihedral groups. The results presented here as
Lemmas \ref{MasseyProdD}, \ref{MasseyProdG}, \ref{MasseyQ} and
\ref{MasseySD} can be summarized
in the following theorem:
\begin{theorem}
Consider $D_{2^n}$ $(n\geq 2)$, $Q_{2^n}$ $(n\geq 3)$ and
$SD_{2^n}$ $(n\geq 4)$ the
dihedral, quaternion and semidihedral groups of
order $2^n$, and the generators of their cohomology as denoted in
Equation \eqref{notationHG}. Then:
\begin{enumerate}[\rm (a)]
\item Neither of the elements $W$, $W+Y^2$ or $W+X^2$ is contained in an iterated Massey product
of degree one elements
in the cohomology of $D_{2^n}$ of order less than $2^{n-1}$.
\item The $2^{n-1}$-fold iterated Massey product in the cohomology of $D_{2^n}$
defined as
$
\langle
X,X+Y,\ldots,X,X+Y
\rangle
$  contains $W$, $W+Y^2$ and $W+X^2$.
\item The $2^{n-1}$th-fold iterated Massey product in the cohomologies of
$D_{2^n}$, $Q_{2^n}$ and $SD_{2^n}$ defined as
$
\langle
X,X+Y,\ldots,X,X+Y
\rangle
$
does not contain the zero element.
\item The $m$-fold iterated Massey product
$\langle
X,X+Y,X,X+Y,\ldots
\rangle
$
is not defined for $m> 2^{n-1}$ in any of the cohomologies of
$D_{2^n}$, $Q_{2^n}$ and $SD_{2^n}$.
\item $\langle Y,Y^2,Y,Y^2\rangle = \{V\}$ in the cohomology of $Q_{2^n}$.
\item $\langle X,X^2,Y \rangle = \{U, U+Y^3\}$ and $\langle X,X^2,X,X^2
\rangle=\{V, V+YU\}$
in the cohomology of $SD_{2^n}$.
\end{enumerate}
\end{theorem}
The proof of this theorem will be done in the next two subsections: the first one is
devoted to the computations of iterated Massey products of degree one elements while the second
will use the Yoneda cocomplex to prove last two statements in the theorem.

\subsection{Iterated Massey products of degree one elements}
The representations described in Appendix \ref{sec:representations} allows us to compute some
iterated Massey
products in the cohomology of the dihedral, quaternion and semidihedral groups:

\begin{lemma}\label{MasseyProdD}
Consider the cohomology of $D_{2^n}$, $n\geq 3$ as denoted in Equation
\eqref{notationHG}.
Then:
\begin{enumerate}[\rm (a)]
\item Neither of the elements $W$, $W+Y^2$ or $W+X^2$ is contained in an iterated Massey product
of degree one elements
of order less than $2^{n-1}$.
\item The $2^{n-1}$-fold iterated Massey product $
\langle
X,X+Y,\ldots,X,X+Y
\rangle
$ contains $W$, $W+Y^2$ and $W+X^2$.
\end{enumerate}
\end{lemma}
\begin{proof}
Assume first that an $m$-fold iterated Massey product contains either
$W$, $W+Y^2$ or $W+X^2$. Then we would have the following diagram:
$$
\xymatrix{&  \Z/2 \ar[d] \ar@{=}[r]&  \Z/2 \ar[d] \\
K \ar[r] \ar@{=}[d]&  G \ar[d] \ar[r]&  U(2,m+1) \ar[d] \\
K \ar[r]&  D_{2^n} \ar [r]&  \overline{U}(2,m+1)
}
$$
where the bottom right square is a pull-back, $K$ is defined as the
kernel of the horizontal arrows and the vertical
lines are central extensions.

As either $W$, $W+Y^2$ or $W+X^2$ classifies the extension, then $G$ is
isomorphic to
either $D_{2^{n+1}}$, $Q_{2^{n+1}}$ or $SD_{2^{n+1}}$.
The center of $G$ is exactly $\Z/2$. If $K$ is non trivial, as it is a
normal subgroup in $G$, then
$K$ intersects non-trivially the center of $G$, so it contains the
center of $G$. But this implies
that the center of $G$ maps injectively to $D_{2^n}$, and it
contradicts the exactness of the
vertical line. This implies that $K$ is trivial, so, there is an
injection of $G$ in $U(2,m+1)$, so
by Lemma \ref{minima_m}, $m\geq 2^{n-1}$.

The representations in Lemma \ref{representations} and Theorem
\ref{theoremdwyer} tell us that
the $2^{n-1}$-fold iterated Massey product
$
\langle
X,X+Y,\ldots,X,X+Y
\rangle
$ contains $W$, $W+Y^2$ and $W+X^2$.
\end{proof}

The following Lemma will use the same notation for the generators of
$H^1(G)$ for different $G$, as noted
in Equation \eqref{notationHG} and
the result applies to all of them:
\begin{lemma}\label{MasseyProdG}
Consider the cohomology of the dihedral, quaternion and semidihedral
groups of order $2^n$
as denoted in Equation \eqref{notationHG}.
Then, for all these groups:
\begin{enumerate}[\rm (a)]
\item The $2^{n-1}$-fold iterated Massey product defined as
$
\langle
X,X+Y,\ldots,X,X+Y
\rangle
$
does not contain the zero element.
\item The $m$-fold iterated Massey product
$\langle
X,X+Y,X,X+Y,\ldots
\rangle
$ is not defined for $m> 2^{n-1}$.
\end{enumerate}
\end{lemma}
\begin{proof}
To prove (a), assume that the zero element is in a $2^{n-1}$-fold iterated Massey product of type
$
\langle
X,X+Y,\ldots,X,X+Y
\rangle
$ of the cohomology of $G$, where $G$ is either $D_{2^n}$, $Q_{2^n}$ or
$SD_{2^n}$.
Then there would be a group morphism from $G$ to $\overline{U}(\F_2,2^{n-1}+1)$
which lifts to a group morphism from $G$ to $U(\F_2,2^{n-1}+1)$ such that the
image of $x$ is a matrix with all the entries in position $(i,i+1)$
equals $1$. The order of
such an element is $2^n$, bigger than the order of $x$, getting a contradiction.

(b) can be deduced from (a): if a $m$-fold iterated Massey product $\langle
X,X+Y,X,X+Y,\ldots
\rangle
$ with $m> 2^{n-1}$ is defined, then the zero element must
be in all the strictly shorter subproducts,
in particular in the $2^{n-1}$-fold product of this type.
\end{proof}

\subsection{Iterated Massey products in higher degrees}
Till now all the computations of iterated Massy products have involved only elements in cohomology in degree $1$, obtaining
elements in degree $2$. In this subsection we compute Massey products involving also elements
in degree $2$, and obtaining the generators in degree $3$ and $4$ which appear in the cohomology
of the quaternion and semidihedral groups.

We start with the quaternion groups.

\begin{lemma}\label{MasseyQ}
Consider the notation in Equation \eqref{notationHG} for the
cohomology of $Q_{2^n}$, the quaternion
group of order $2^n$ with  $n\geq 3$.
Then
$$
\langle Y,Y^2,Y,Y^2\rangle = \{V\} .
$$
\end{lemma}
\begin{proof}
In this proof we will consider $(P_\bullet,\partial_\bullet)$, the projective resolution of 
$\F_2$ as $\F_2[Q_{2^n}]$-module
described in Lemma \ref{lemmaprojresQ}. We also consider the generators $Y$ and $V$ in $H^*(BQ_{2^n})$ 
as elements in the Yoneda cocomplex as
computed in Lemma \ref{lemma:genHQ}.

With all this data, computing products in these generators
can be carried out by composing the corresponding cochain maps, so we can easily find cochain maps
representing $Y^2$ and $Y^3$:
\begin{enumerate}[(i)]
 \item $Y^2$ can be represented by cochain maps $(Y^2)_i\definicio Y_{i+1}\circ Y_i \colon P_{i+2} \to P_i$, obtaining:
$$
(Y^2)_{4i}= (1 \, 0), \,
(Y^2)_{4i+1}= \left( \begin{smallmatrix} 1 \\ 0 \end{smallmatrix} \right), \,
(Y^2)_{4i+2}= \left( \begin{smallmatrix} 0 \\ N_x \end{smallmatrix} \right)
\text{ and }
(Y^2)_{4i+3}= (0 \, N_x).
$$
 \item To compute the cochain maps representing $Y^3$ we can use the previous ones in the composition
$(Y^3)_i\definicio (Y^2)_{i+1}\circ Y_i \colon P_{i+3} \to P_i$, obtaining:
$$
(Y^3)_{4i}= (0), \,
(Y^3)_{4i+1}= \left( \begin{smallmatrix} N_x \\ 0 \end{smallmatrix} \right), \,
(Y^3)_{4i+2}= \left( \begin{smallmatrix} 0&0\\ 0&N_x \end{smallmatrix} \right)
\text{ and }
(Y^3)_{4i+3}= (N_x\, 0).
$$

\end{enumerate}
Now we proceed computing one element in the $4$-fold iterated Massey product
$\langle Y,Y^2,Y,Y^2 \rangle$. We must find the following coefficients in a defining system:
$$
\begin{pmatrix}
1&  Y&  \alpha&  \beta&  \\
0&  1&   Y^2&  \alpha&  \gamma \\
0&  0&    1&    Y&  \alpha \\
0&  0&    0&    1&   Y^2 \\
0&  0&    0&    0&    1
\end{pmatrix}
$$
such that $\delta \alpha=Y^3$, $\delta \beta = Y \alpha + \alpha Y$,
$\delta \gamma = Y^2\alpha + \alpha Y^2$. Remark that we just need to compute the
maps in low degrees, just enough to compose them and get the product. Now we proceed in the computation
of representatives for $\alpha$, $\beta$ and $\gamma$:
\begin{enumerate}[(i)]
 \item To compute $\alpha$ we start with $\alpha_0=(1\quad 0) \colon P^2 \to P^0$. Now use the relation
$(Y^3)_0=(\delta \alpha)_0=\alpha_0 \circ \partial_3 - \partial_1\circ \alpha_1$, so
$\alpha_1$ is a lifting of the $\F_2[Q_{2^n}]$-morphism $(Y^3)_0 + \alpha_0 \circ \partial_3$ (we can avoid the
signs as we are working in $\F_2$). This can be
done step by step obtaining $\alpha_i \colon P_{i+2} \to P_i$:
$$
\alpha_0=(1\, 0), \,
\alpha_1=\left(\begin{smallmatrix} 1 \\ 0 \end{smallmatrix}\right),\,
\alpha_2=\left(\begin{smallmatrix} J \\ L \end{smallmatrix}\right)
\text{ and }
\alpha_3=(K\quad L+N_x).
$$
 \item Computation of $\beta$ is done from the formula $\delta \beta = Y \alpha + \alpha Y$, so we need first
$\F_2[Q_{2^n}]$-morphisms $(Y \alpha + \alpha Y)_i\colon P_{i+3} \to P_i$ calculated as composition and sum of the
previously computed $Y_j$ and $\alpha_k$, obtaining:
$$
\begin{matrix}
(Y \alpha + \alpha Y)_0=(0), \,
(Y \alpha + \alpha Y)_1=\left(\begin{smallmatrix} L+N_x \\ J \end{smallmatrix}\right),\,
(Y \alpha + \alpha Y)_2=\left(\begin{smallmatrix} 0&J \\ K&N_x \end{smallmatrix}\right),\,
\\
(Y \alpha + \alpha Y)_3=(L+N_x \quad K)
\text{ and }
(Y \alpha + \alpha Y)_4=(0).
\end{matrix}
$$ 
So we can now proceed to compute a representative for $\beta$ also lifting step by step the corresponding
morphisms, using the same
technique as the used in the computation in $\alpha$. $\beta_i \colon P_{i+2} \to P_i$ can be taken as:
$$
\beta_0=(0\, 0), \,
\beta_1=\left(\begin{smallmatrix} 0 \\ 0 \end{smallmatrix}\right), \,
\beta_2=\left(\begin{smallmatrix} 1+J+J^2 \\ L(1+J) \end{smallmatrix}\right)
\text{ and }
\beta_3=(K\, L+N_x).
$$
\item Relation $\delta \gamma = Y^2\alpha + \alpha Y^2$ gives us the coefficients in the computation
of $\gamma$, following the same procedure as the computation of $\beta$. For the sake of completeness we
list here $(Y^2\alpha + \alpha Y^2)_i \colon P_{i+4} \to P_i$ for small $i$:
$$
\begin{matrix}
(Y^2\alpha + \alpha Y^2)_0=(J), \,
(Y^2\alpha + \alpha Y^2)_1=\left(\begin{smallmatrix} K&L \\ 0&0 \end{smallmatrix}\right),\,
(Y^2\alpha + \alpha Y^2)_2=\left(\begin{smallmatrix} J&0 \\ L&0 \end{smallmatrix}\right)
\\
\text{ and }
(Y^2\alpha + \alpha Y^2)_4=(K).
\end{matrix}
$$
And also some possible choices for $\gamma$. $\gamma_i \colon P_{i+3} \to P_i$ can be taken as:
$$
\gamma_0=(0), \,
\gamma_1=\left(\begin{smallmatrix} 0 \\ 1 \end{smallmatrix}\right), \,
\gamma_2=\left(\begin{smallmatrix} 0 & 1 \\ 1 & 0 \end{smallmatrix}\right)
\text{ and }
\gamma_3=(0\, 1).
$$
\end{enumerate}
Now we can get an element of the $4$-fold iterated Massey product composing and adding the previous computations
in the formula $Y\gamma+\alpha^2  + \beta Y$. We obtain the $\F_2[Q_{2^n}]$-morphism 
$(Y\gamma+\alpha^2  + \beta Y)_0 \colon P_4 \to P_0$ represented by the matrix $(1)$, so it is
equivalent to $V$.

This procedure proves that $V \in \langle Y,Y^2,Y,Y^2\rangle$, but
we have done several choices.
Let us see now that if we choose other elements we get again $V$: by
\cite[Theorem 3]{Kraines}
we can fix the representatives of $Y$ and $Y^2$ to construct any
defining system. Assume we change
all coefficients $\alpha$ by (possibly) $\alpha'$, $\alpha''$,
$\alpha'''$, $\beta$ by $\beta'$,
and finally $\gamma$ by $\gamma'$ getting a new defining system:
$$
\begin{pmatrix}
1&  Y&  \alpha'&  \beta'&  \\
0&  1&   Y^2&  \alpha''&  \gamma' \\
0&  0&    1&    Y&  \alpha''' \\
0&  0&    0&    1&   Y^2 \\
0&  0&    0&    0&    1
\end{pmatrix} .
$$
We have that $\delta(\alpha - \alpha')=0$, so there are $a'$ and $b'$ in $\F_2$ such that
$$
\alpha'=\alpha + a'X^2+b'Y^2.
$$
The same argument applied to $\alpha''$ and $\alpha'''$ gives us
the relations 
$$\alpha''=\alpha + a''X^2+b''Y^2 
\text{ and }
\alpha'''=\alpha + a'''X^2+b'''Y^2,
$$ 
for $a''$, $a'''$, $b''$ and $b'''$ elements in $\F_2$.

Now $\delta(\beta-\beta')=(Y\alpha''+\alpha'Y)-(Y\alpha+\alpha Y)=
Y(\alpha''-\alpha)+(\alpha'-\alpha)Y=Y(a''X^2+b''Y^2)+(a'X^2+b'Y^2)Y=(a''+a')YX^2+(b''+b')Y^3$. As
$YX^2\neq 0 \in H^*(BQ_{2^n})$ we get that $a''=a'$ and there are possibly $a^{(iv)}$,
$b^{(iv)}$ and $c^{(iv)}$ in $\F_2$ such that 
$$
\beta'= \beta + a^{(iv)}X^2 + b^{(iv)}Y^2 + c^{(iv)} \alpha .
$$
A similar expansion of the expression $\delta(\gamma-\gamma')=(Y\alpha'''+\alpha''Y)-(Y\alpha+\alpha Y)$
tells us that
$$
\gamma'=\gamma + a^{(v)}X^3+b^{(v)}Y^3.
$$
So, the result of
this defining system differs 
from $V$ by an element which can be written as, after simplifying,
$\alpha (a'''X^2+(b'''+1)Y^2)+(a'X^2+b'Y^2)\alpha$, for $a',a''',b',b''' \in \F_2$. We
use again the Yoneda
cocomplex for all this generators to compute it, and this always give
the zero element in cohomology, so
$$
\langle Y, Y^2, Y , Y^2 \rangle = \{V\} .
$$
\end{proof}

\begin{lemma}\label{MasseySD}
Let $SD_{2^n}$ be a semidihedral group of order $2^n$, with $n\geq 4$.
Fix $X$, $Y$, $U$ and $V$ the
generators in $H^*(SD_{2^n})$ as in Equation \eqref{notationHG}. Then
$$
\langle X,X^2,Y \rangle = \{U, U+Y^3\}
\text{ and }
\langle X,X^2,X,X^2 \rangle=\{V, V+YU\}.
$$
\end{lemma}
\begin{proof}
Here we use the projective resolution $(P_\bullet, \partial_\bullet)$ from Lemma \ref{lemmaprojresSD}
and the properties of the representatives of $X$, $Y$, $U$ and $V$ in the Yoneda listed in
Lemma \ref{lemma:genHSD}.

The thread of the proof is the same as the one in Lemma \ref{MasseyQ}. 
To make it shorter and clearer we will focus on the entries of the matrices as
$\F_2[Q_{2^n}]$-morphisms in the Yoneda cocomplex which define
the elements we are interested in.

We must start checking that $U$ or $U+Y^2$ are in $\langle X,X^2,Y\rangle$. 
By Lemma \ref{lemma:genHSD}, this reduces to see that
the first coordinate, as $\F_2[Q_{2^n}]$-morphism $P_3 \to P_0$ in the Yoneda cocomplex, 
in a explicit defining system of this iterated Massey product is non zero. We need to compute the
first degrees of cochain maps $\alpha$ and $\beta$ fitting in the following defining system:
$$
\begin{pmatrix}
1 & X & \alpha & \\
0 & 1 & X^2 & \beta \\
0 & 0 & 1& Y
\end{pmatrix} .
$$
\begin{enumerate}[(i)]
\item The $\F_2[SD_{2^n}]$-morphism $\alpha$  must fulfill 
$\delta \alpha=X^3$. A direct computation
gives us that we can take $\alpha_i\colon P_{i+2}\to P_{i}$ as: 
$$
\alpha_0=(1\, 0\, 0) \text{ and } \alpha_1=\left(\begin{matrix}
1+I^{2^{n-1}-3} & t(L+I^{2^{n-2}-2}) & x^2I^{2^{n-1}-5} & 0 \\ 0 &
1 & 0 & 0
\end{matrix} \right) .
$$
\item Observe that $\delta$ must fulfill $\delta \beta = X^2 Y$ and  we have the relation 
$X^3=X^2Y$, so we can take $\alpha=\beta$. 
\end{enumerate}
So the resulting element in $\langle X, X^2,Y\rangle$ of this defining system is 
$X \alpha + \alpha Y=X \alpha + \beta Y$, and we just need to compute the evaluation
of the morphism at level $0$:
$$
\varepsilon (\alpha_0 X_2 + Y_0\alpha_1)=(1\,0\,1\,0) ,
$$ 
as a map from $P_3$ to $\F_2$. By Lemma \ref{lemma:genHSD}, this element corresponds either to $U+Y^3$ or $U$.

Now we must deal with the indeterminacy in this $3$-fold iterated Massey products
to get the final result.
Applying \cite[Proposition 2.3]{May}, in a $3$-fold iterated Massey product, two
elements in
$\langle X , X^2, Y \rangle$ differ by an element of the form
$\langle Z,Y \rangle + \langle X, Z' \rangle$, for $Z, Z'$ elements in
$H^2(BSD_{2^n})$. As
$Y^3$ is the only element which can be constructed in this way we get
$$
\langle X,X^2,Y \rangle = \{ U, U+Y^3\} .
$$

We need now to compute one element in $\langle X,X^2,X,X^2\rangle$. 
Again, the computations will focus on the information which tell us that
either $V$ or $V+YU$ belongs to this $4$-fold iterated Massey product: that is, by Lemma 
\ref{lemma:genHSD}, there is an element
of the form $(1\, \ast\, \ast\, \ast\, 0)$ in $\langle X,X^2,X,X^2\rangle$, seen
as a map from $P_4$ to $\F_2$ in the Yoneda cocomplex.

Here we use the same procedure as in the
proof of Lemma \ref{MasseyQ}, using the matrices $X_i$ described above. We need $\alpha$,
$\beta$ and $\gamma$ in a defining system:
$$
\begin{pmatrix}
1&  X&  \alpha&  \beta&  \\
0&  1&   X^2&  \alpha&  \gamma \\
0&  0&    1&    X&  \alpha \\
0&  0&    0&    1&   X^2 \\
0&  0&    0&    0&    1
\end{pmatrix}.
$$
So the relations we must care about are $\delta \alpha=X^3$,
$\delta\beta=X\alpha+\alpha X$ and $\delta\gamma=X^2\alpha+\alpha X^2$.

By definition, $\alpha$ may be taken the same $\alpha$ considered
in the computation of $\langle X,X^2,Y\rangle$. Moreover, as
$X\alpha+\alpha X$ gives the same representative as $X^3$ in the
Yoneda cocomplex, we can take also $\beta=\alpha$. Finally, this
election of $\alpha$ makes $X^2\alpha+\alpha X^2$ to be the zero
element, so we can take $\gamma=0$.

Now we can proceed in the computation of the first coordinate of
$X\gamma + \alpha^2 + \beta X^2$:
$$
\varepsilon (0+\alpha_0 \alpha_2+X_0X_1\alpha_2) .
$$
We just need the first and last columns of $\alpha_2$, which we can see that have $(1,0,0)$ and $(0,0,0)$
as coefficients respectively. As the result of the computations is the sum of the first two rows,
we get that the result is of the form $(1 \, \ast \, \ast \, \ast\, 0)$ as
a map from $P_ 4$ to $\F_2$, obtaining either $V$ or $V+YU$ by Lemma \ref{lemma:genHQ}.

We can also use the computations in the indeterminacy of the $4$-fold Massey iterated product in Lemma 
\ref{MasseyQ} to get a formula for the indeterminacy in this $4$-fold iterated Massey product. 
Two elements in $\langle
X,X^2,X,X^2 \rangle$ differ
by an element which can be expressed as
$\alpha(aX^2+bY^2)+(a'X^2+b'Y^2)\alpha$, with
$\alpha$ a cochain such that $\delta \alpha = X^3$ and $a,b,a',b' \in
\F_2$. Any element
of this form gives a map $\F_2[G]^5 \to \F_2[G] \to \F_2$ with zeros in the
first and last
coordinates and the coefficient $a'$ in the second. So just the element
$YU$ is in the indeterminacy,
getting that:
$$
\langle X,X^2,X,X^2 \rangle=\{V, V+YU\}.
$$
\end{proof}


\section{Cohomological uniqueness}
This section is divided in 3 subsections with the same structure, 
one for each family of $2$-groups of maximal nilpotency class. 

\subsection{Dihedral groups}
Consider $D_{2^n}$ the dihedral group of order $2^n$, and its cohomology for $n\geq 3$
as defined in Equations \eqref{notationG} and \eqref{notationHG}.

The following lemma uses the fact that, for any topological space $\X$, $H^*(\X)$ is an unstable
algebra over the Steenrod algebra, but the lemma does not
require that this structure must be the same as the one in $H^*(BD_{2^n})$:
\begin{lemma}\label{isoorzeroD}
Let $\X$ be an space such that $H^*(\X)\cong H^*(BD_{2^n})$ as algebras.
Let $\phi \colon \X \to BD_{2^n}$ be a map inducing the identity in degree one
cohomology. Then either $\phi^*$ is an isomorphism or $\phi^*(W)=0$.
\end{lemma}
\begin{proof}
According to the hypothesis, $\phi^*$ is an isomorphism if and
only if $W$ is in the image of $\phi^*$. In other words $\phi^*$
is not an isomorphism if and only if $\phi^*(W)=aX^2+bY^2$ for
$a,b \in \F_2$. Assume then that $\phi^*(W)=aX^2+bY^2$. Applying
$\Sq{1}$ in both sides we get $\Sq{1}(W)=WY$ and
$\Sq{1}(aX^2+bY^2)=0$. So
$0=\Sq{1}(\phi^*(W))=\phi^*(\Sq{1}(W))=(aX^2+bY^2)Y$ which implies
$a=b=0$.
\end{proof}

\begin{theorem}\label{theoremD}
$BD_{2^n}$ is determined by its $\M$-cohomology.
\end{theorem}
\begin{proof}
Fix $\X$ a $2$-complete topological space having the homotopy type
of a CW-complex and $\M$-comparable to $BD_{2^n}$.

For $n=2$, $D_4\cong\Z/2 \times \Z/2$ and the result follows from
Proposition \ref{example:elementaryabelian}.

Assume that $n \geq 3$. Then we should give a map $\phi_n \colon \X \to BD_{2^n}$ inducing an
isomorphism in cohomology
up to degree $2$. Consider the tower of principal fibrations in Equation \eqref{fibration},
where each $\pi_k$ corresponds to the central extension:
$$
1 \to \Z/2 \to D_{2^{k+1}} \to D_{2^k} \to 1
$$
classified either by $X^2+XY$ if $k=2$ or by $W$ when $k>2$.

Consider $\phi_2 \colon \X \to BD_4$, a map representing the classes $X$ and $Y$. The composite
$$
\X \stackrel{\phi_2}{\longrightarrow}BD_4 \stackrel{X^2+XY}{\longrightarrow} K(\Z/2,2)
$$
is nullhomotopic, so $\phi_2$ factors as a composition $\pi_2 \circ \phi_3$, with
$\phi_3 \colon \X \to BD_8$.

Now, if we assume that $\phi_k \colon \X \to BD_{2^k}$ inducing
the identity in $H^1$ is defined, this map will extend to a map
$\phi_{k+1} \colon \X \to BD_{2^{k+1}}$ if and only if
$\phi_k^*(W)=0$. Using Lemmas \ref{minima_m} and \ref{isoorzeroD}, and
the fact that $(2^k-2)$-fold iterated Massey products must fulfill
$$ 
\phi_k^*(\langle X,X+Y,X,X+Y,\cdots \rangle_{H^*(BD_{2^k})})\subset \langle
X,X+Y,X,X+Y, \cdots \rangle_{H^*(\X)}
$$ 
we get
that $\phi_k^*(W)=0$ for all $k< n$ (the subscripts in the formula
indicate the algebra where the $(2^k-2)$-fold iterated Massey products are
considered). So, the map $\phi_k$ extends to $\phi_{k+1}\colon \X
\to BD_{2^{k+1}}$ for $k<n$.

It remains to check the last step: if $\phi_n(W)=0$, then it extends to $\phi_{n+1}$.
Such a map $\phi_{n+1} \colon \X \to BD_{2^{n+1}}$ inducing
the identity in $H^1$ cannot exist because the $2^{n+1}$-fold iterated Massey product 
$\langle X,X+Y,X,X+Y,\cdots\rangle$
is defined in
$H^*(BD_{2^{n+1}})$ and it is not defined in $H^*(\X)$ by Lemma \ref{minima_m}.

So $\phi_{n}^*(W)\neq 0$, and by Lemma \ref{isoorzeroD},
$\phi_n$ is an isomorphism in cohomology and $\X\simeq BD_{2^n}$.
\end{proof}


\subsection{Quaternion groups}

\begin{lemma}\label{imagew}
Consider the notation in Equation \eqref{notationHG} for the cohomology of the dihedral
and quaternion groups.
Let $\X$ be an space $\M$-comparable to $BQ_{2^n}$, and $\phi \colon \X \to
BD_{2^r}$ be a map such that $\phi^*(X)=X$ and $\phi^*(Y)=Y$. Then $\phi^*(W)=0$ or $\phi^*(W)=Y^2$.
\end{lemma}
\begin{proof}
Recall $\{X^2,Y^2\}$ is a basis of $H^2(\X)$. Then there are
$a,b\in\F_2$ such that $\phi^*(W)=aX^2+bY^2$. If we apply $\Sq{1}$
to both sides of the identity we get $a X^2Y =0$ (notice $Y^3=0$
in $H^*(\X)$), getting the desired result.
\end{proof}

\begin{theorem}\label{theoremQ}
$BQ_{2^n}$ is determined by its $\M$-cohomology.
\end{theorem}
\begin{proof}
Let $\X$ be a $2$-complete space having the homotopy type of a
CW-complex and $\M$-comparable to $BQ_{2^n}$.

We must consider the cases $n=3$ and $n\neq 3$ separately.

Consider $Q_8$ and its cohomology as in Equations \eqref{notationG} and \eqref{notationHG}.
Let $\phi_2 \colon \X \to B(\Z/2 \times
\Z/2)$ be the map representing the elements $X$ and $Y$ in cohomology. Such a map factors
through $\phi_3\colon \X \to Q_8$ because $\phi_2^*(X^2+XY+Y^2)=0$, so we have a map $\phi_3$ inducing
the identity in $H^1$. Use now Lemma \ref{MasseyQ} to get that it must be an
isomorphism in cohomology, so a homotopy equivalence.

Assume now $n\geq 4$ and that $\X$ is a $\F_2$-complete space
$\M$-comparable to $BQ_{2^n}$. Let $\phi_2 \colon \X \to B(\Z/2
\times \Z/2)$ be the map representing the elements $X$ and $Y$ in
cohomology. Consider the tower of principal fibrations $$ \cdots
\to BD_{2^k} \stackrel{\pi_{k-1}}{\to}BD_{2^{k-1}}
\stackrel{\pi_{n-2}}{\to} \cdots \stackrel{\pi_{3}}{\to} BD_8
\stackrel{\pi_{2}}{\to} B(\Z/2\times\Z/2).$$ Recall that the map
$\pi_2$ is classified by the class $X^2+XY$ and each $\pi_i$ for
$i\geq 3$ is classified by $W$.

Since $\phi_2^*(X^2+XY)=0$, $\phi_2$ lifts to $\phi_3\colon \X \to BD_8$, a map
such that $\phi_3^*(X)=X$ and $\phi_3^*(Y)=Y$.


Assume now that $k\geq 3$ and $\phi_k\colon \X \to BD_{2^k}$ such
that  $\phi_k^*(X)=X$, $\phi_k^*(Y)=Y$. We show that
$\phi_k^*(W)=0$ when $k<n-1$, which implies that $\phi_k$ lifts to
$\phi_{n-1}\colon \X \to BD_{2^{n-1}}$ such that
$\phi_{n-1}^*(X)=X$, $\phi_{n-1}^*(Y)=Y$.

If $\phi_k^*(W)\neq 0$, then $\phi_k^*(W)=Y^2$ by Lemma \ref{imagew}.
This implies $\phi_k^*(Y^2+W)=0$ and there is a map $\tilde
\phi_{k+1} \colon \X \to BQ_{2^{k+1}}$ which is an isomorphism in
cohomology up to degree $3$. This implies, by Lemma \ref{lemmaMay},
that both $\X$ and $BQ_{2^{k+1}}$ have the same iterated Massey
products involving degree one elements. But by Lemma
\ref{MasseyProdG}, as $k+1<n$ there are iterated Massey products
in $H^*(\X)$ which are not defined in $H^*(BQ_{2^{k+1}})$, getting
a contradiction to the assumption $\phi_k^*(W)\neq 0$.

It remains to see that $\phi_{n-1}$ lifts to $\phi_n\colon \X \to
BQ_{2^n}$, i.e.\ that $\phi_{n-1}(W)=Y^2$: if $\phi_{n-1}(W)\neq
Y^2$ then $\phi_{n-1}(W)=0$ by Lemma \ref{lemmaMay}, and we 
get a map $\tilde \phi_n \colon \X \to D_{2^n}$. If such a
$\tilde\phi_n$ exists, using again by Lemma \ref{lemmaMay}, there
exists either a map $\tilde\phi_{n+1} \colon \X \to
D_{2^{n+1}}$ or a map $\tilde\phi'_{n+1} \colon \X \to
Q_{2^{n+1}}$, but neither $\tilde\phi_{n+1}$ nor
$\tilde\phi'_{n+1}$ can exist because in the cohomology of both
targets there are $2^n$-fold iterated Massey products of type $\langle X,X+Y,X, \dots ,
X+Y\rangle$ which are not defined in $H^*(\X)$ by
Lemma \ref{MasseyProdG}, getting a contradiction.

This implies that we have a map $\phi_n \colon \X \to BQ_{2^n}$
which is the identity in cohomology till degree $3$, and by Lemma
\ref{MasseyQ}, $\phi_n^*(V)=V$. Then $\phi_n^*$ is an isomorphism
and $\phi_n$ is a homotopy equivalence.
\end{proof}

\begin{remark} \label{remark:higherneeded}
Observe that iterated Massey products involving elements in degree greater than one cannot be avoided.
Consider, as in \cite{Benson-BCAT90}, $\X=S^3/Q_{2^n} \times BS^3$ where the action of
$Q_{2^n}$ is by left multiplication
on $S^3$ where one considers $Q_{2^n}$ as a discrete subgroup of the Lie group $S^3$.

The space $\X$ is $2$-good, and using \cite{Benson-BCAT90}, we have that $H^*(\X)\cong H^*(BQ_{2^n})$
as $\F_2$-algebras. Moreover, the iterated products of degree one elements are the same: all the
information about the iterated Massey products of elements in degree one can be read in the three first
steps of a minimal projective resolution of $\F_2[Q_{2^n}]$, and the three steps agree
with the minimal
projective resolution of $C^*(\X,\F_2)$.  Finally, these are not homotopy equivalent spaces up to $2$ completion
because, for example, they have different homotopy groups.
\end{remark}

\subsection{Semidihedral groups}
Consider $SD_{2^n}$ a semidihedral group of order $2^n$ and its cohomology with the notation in
Equations \eqref{notationG} and \eqref{notationHG}.
\begin{lemma}\label{imagewSD}
Let $\X$ be a topological space $\M$-comparable to $BSD_{2^n}$, and $\phi \colon \X \to
BD_{2^r}$ ($r\geq 3$) be a map such that $\phi^*(X)=X$ and $\phi^*(Y)=Y$. Then $\phi^*(W)=0$ or
$\phi^*(W)=X^2$.
\end{lemma}
\begin{proof}
Consider $\{X^2, Y^2\}$ as basis of $H^2(\X)$ as $\F_2$ vector space. Then
there are $a$ and $b$ in $\F_2$ such that $\phi^*(W)=aX^2+bY^2$. If we apply
$\Sq{1}$ to both sides of the identity we get $b Y^3 =0$, getting the desired
result.
\end{proof}

\begin{theorem}\label{theoremSD}
$BSD_{2^n}$ is determined by its $\M$-cohomology.
\end{theorem}

\begin{proof}
Fix $\X$ a $2$-complete space having the homotopy type of a $CW$-complex and
$\M$-comparable to $BSD_{2^n}$ $(n\geq 4)$.

Consider $\phi_2\colon \X \to BD_4$ a map representing $X$ and $Y$. As $\phi_2^*(X^2+XY)=0$, this map will
factor through a map $\phi_3\colon \X \to BD_8$ such that $\phi^*_3\colon H^1(BD_8) \to H^1(\X)$
is the identity.

Assume now that we have a map $\phi_k \colon \X \to BD_{2^k}$ which induces the identity in $H^1$ and
with $k<n-1$. Then, by Lemma \ref{imagewSD}, $\phi^*_k(W)=X^2$ or $\phi^*_k(W)=0$.

In the first case,
$\phi^*_k(W+X^2)=0$, so there is a map $\tilde\phi_{k+1}\colon \X \to BSD_{2^{k+1}}$ which is the
identity
in cohomology in degrees one and two. So, by Lemma \ref{lemmaMay}, they must have the same iterated Massey
products of degree one elements. But, by Lemma \ref{MasseyProdG}, if $k+1<n$ there are
$2^{n-1}$-fold iterated Massey products which are defined in the cohomology of $\X$, but not in
$H^*(BSD_{2^{b+1}})$, and this contradicts Lemma \ref{lemmaMay}.

So we are in the second case and we have
a map $\phi_{k+1} \colon \X \to BD_{2^{k+1}}$ inducing the identity in $H^1$.
This procedure can be done up to
$\phi_{n-1} \colon \X \to BD_{2^{n-1}}$ inducing the identity in $H^1$.
Again, by Lemma \ref{imagewSD}, $\phi^*_{n-1}(W)=0$ or $\phi^*_{n-1}(W)=X^2$.

In the first case, we obtain a map $\tilde\phi_n \colon \X \to BD_{2^n}$ inducing the identity
in $H^1$. By the previous
arguments such a map induces a map either $\tilde\phi_{n+1}\colon \X \to BD_{2^{n+1}}$ or
$\tilde\phi_{n+1}\colon \X \to BSD_{2^{n+1}}$ inducing the identity in $H^1$, and this cannot happen
because Lemmas \ref{MasseyProdD}
and \ref{MasseyProdG} imply that there are iterated Massey products of degree one elements defined in
$H^*(BD_{2^{n+1}})$ and $H^*(BSD_{2^{n+1}})$ which are
not defined in $H^*(\X)$.

So $\phi^*_{n-1}(W)=X^2$ and we have a map $\phi_n \colon \X \to BSD_{2^n}$ inducing the identity in $H^1$.
Now we use now Lemma \ref{MasseySD} to see that it must be an isomorphism in cohomology, and therefore a homotopy
equivalence: as $\phi_n^*(\{U,U+Y^3\})=\phi_n^*(\langle X,X^2,Y\rangle)\subset
\langle X,X^2,Y \rangle = \{U,U+Y^3\}$ we get that either $\phi_n^*(U)=U$ or $\phi_n^*(U+Y^3)=U$, so
$U$ is in the image of $\phi_n^*$. Using the same argument applied
to $\langle X,X^2,X,X^2\rangle = \{V,V+YU\}$
we get that $V$ is in the image of $\phi_n^*$. This implies that all generators are in the image, and, up
to degree $4$ all are finite dimensional vector spaces, so an epimorphism is an isomorphism.
\end{proof}

\appendix
\section{Representations of maximal nilpotency class 2-groups}\label{sec:representations}
In this section we obtain an
explicit minimal degree faithful representation of maximal nilpotency
class $2$-groups on $\overline{U}(\F_2,n)$ which will allow us to compute
iterated Massey products of some degree
one elements.
The definition of the images of the generators of each group is
defined using the matrices described below:

Consider $A_n$ the $2^n\times 2^n$ matrix defined inductively:
\begin{equation}\label{defA}
A_0=(1) \quad , \quad A_n=\left(\begin{array}{c|c} A_{n-1}&  A_{n-1} \\ \hline
0&  A_{n-1} \end{array}
\right) ,
\end{equation}
where $0$ means a matrix with all the entries equal zero.

Consider also the $2^n \times 2^n$-matrix $B_n$ with entries $b_{i,j}$:
\begin{equation}\label{defB}
b_{i,j}\definicio \left\{\begin{array}{ll}
1&  \mbox{if $i=j$ or $j=i+1$,} \\ 0&  \mbox{otherwise.}
\end{array} \right.
\end{equation}
Finally consider the $2^n \times 2^n$-matrix $C_n$ with entries $c_{i,j}$:
\begin{equation}\label{defC}
c_{i,j}\definicio \left\{\begin{array}{ll}
1&  \mbox{if $i=j$,} \\ 1&  \mbox{if $i=1$ and $j=2^n$,} \\
0&  \mbox{otherwise.}
\end{array} \right.
\end{equation}
Most of the proofs will be done by induction, so we need to express matrices $B_n$ and
$C_n$ as a construction of $2^{n-1}\times2^{n-1}$ matrices: define the $2^n\times 2^n$
matrix $\Delta_n$ as the matrix with al entries $0$ but the the entry in position
$(1,2^n)$ which is $1$. Then:
\begin{equation}\label{defBC}
B_n=\left(\begin{array}{c|c} B_{n-1}&  \Delta_{n-1}^T \\ \hline
0&  B_{n-1} \end{array} \right)
\text{ and }
C_n=\left(\begin{array}{c|c} \Id&  \Delta_{n-1} \\ \hline
0&  \Id \end{array} \right) ,
\end{equation}
where $\Delta_{n-1}^T$ means the transpose of $\Delta_{n-1}$ and $\Id$ the identity matrix in
the corresponding rank.

The following calculation gives the property needed to compute the iterated Massey
products:
\begin{lemma}\label{lemmaorderB}
Let $B$ be a matrix in $\GL_{m}(\F_2)$ as defined in Equation \eqref{defB} (now
$m$ is not necessarily of the form $2^n$).
Then:
\begin{enumerate}[\rm (a)]
 \item The order of $B$ is $2^n$, where $2^{n-1}<m\leq 2^n$.
 \item If $l$ is a positive integer such that $2^l<\order(B)$, then
coefficients in the diagonal and
in positions $(i,i+2^l)$ in $B^{2^l}$ are equal to $1$ (for $i$
from $1$ to $m-2^l$) and $0$ in all other positions.
\end{enumerate}
\end{lemma}
\begin{proof}
$B$ is the sum of the identity $\Id$ and a nilpotent matrix $N$. Coefficients of $N^a$ are
equal to $0$ in positions $(i,k)$ for $k\neq i+a$, and equal to $1$ in positions $(i,i+a)$. 
In particular $N^a$ is not $0$ for $a<m$ and is $0$ for $a\geq m$. 
Now, as $\Id$ and $N$ commute, we can compute $(\Id+N)^l$ using the mod $2$ binomial formula. 
In particular, $B^{2^l}=(\Id+N)^{2^l}=\Id+N^{2^l}$. Use the description of $N^a$ 
described just before and we get (b). Moreover this tells us that $B^{2^n}=\Id$ and 
that $B^{2^{n-1}}\neq \Id$, where $n$ is the one considered in (a), and we get the desired result.
\end{proof}

\begin{lemma}\label{lemmaABC}
Fix $n\geq 1$. The matrices $A_n$, $B_n$ and $C_n$ defined in Equations
\eqref{defA}, \eqref{defB} and \eqref{defC} have the following
properties:
\begin{enumerate}[\rm (a)]
\item $A_n^2=\Id$,
\item $A_n\Delta_n^TA_n$ is a matrix with all entries equal to $1$,
\item $A_nB_nA_n$ has all the entries over the diagonal equal to $1$.
\item $(A_nB_n)^2=\Id$,
\item the first row of $A_nB_n$ is equal to the vector $(1 \, 0\, \cdots \, 0)$, and
\item $C_n$ is in the center of the invertible upper triangular matrices, so,
in particular, commute with $A_n$ and $B_n$.
\end{enumerate}
\end{lemma}
\begin{proof}
\begin{enumerate}[(a)]
\item This can be seen by induction on $n$ and the fact that all coefficients are taken in $\F_2$: the
case $n=0$ follows directly, and assume the result is true up to $n-1$. Then:
$$
A_n^2=\left(\begin{array}{c|c} A_{n-1}&  A_{n-1} \\ \hline
0&  A_{n-1} \end{array}
\right)^2=
\left(\begin{array}{c|c} A_{n-1}^2&  2A_{n-1} \\ \hline
0&  A_{n-1}^2 \end{array} 
\right) = \Id ,
$$
where we are using that, as  we are working in $\F_2$, $2A_{n-1}=0$.
\item This can proved by induction, proceeding as in (a) and using that:
$$
\Delta_n^T=\left(\begin{array}{c|c} 0 &  0 \\ \hline
\Delta_{n-1}^T &  0 \end{array} 
\right) , 
$$
where, in this case, $0$ is the $2^{n-1}\times2^{n-1}$ matrix with all entries equal $0$.
\item Here we can also prove it by induction: the case $n=0$ is direct. Assume the result is
true up to $n-1$, then:
$$
A_nB_nA_n=\left(
\begin{array}{c|c}
A_{n-1}B_{n-1}A_{n-1} & A_{n-1}\Delta_{n-1}^TA_{n-1} \\ \hline 0 & A_{n-1}B_{n-1}A_{n-1}
\end{array}
\right) .
$$
We can use now the induction hypothesis and (b) to get the result.
\item We can write $(A_nB_n)^2=(A_nB_nA_n)B_n$ and use (c): 
the position $(i,j)$ of the result of this multiplication
is the sum of all entries in the column $j$ of $B_n$ from row $i$ on, so this is
$1$ in the diagonal, and $0$ in all other positions, either because we are adding $0$ or two $1$, which
is $0$ in $\F_2$.
\item Observe that the first row of $A_n$ has all coefficients equal $1$, so the first
row of $A_nB_n$ is, in each position, the sum of all the entries in the corresponding
column, so $1$ in the first position and $0$ in the other.
\item If $U$ is an invertible upper triangular $2^n\times 2^n$-matrix, the
matrix $U C_n$ has the same entries as $U$ but the last element of the first
row, which changes by $+1$. Multiplying $C_n U$ then we
have to add $1$ to the first element of the last column, getting the same
result.
\end{enumerate}
\end{proof}

Now we will use the matrices $A_n$, $B_n$ and $C_n$ to construct the
following $(2^n+1)\times(2^n+1)$
matrices:
$$
x_n\definicio
{\scriptscriptstyle
\left(\begin{array}{c|c}
1&  \begin{array}{cccc}  1&  0&   \cdots&  0 \end{array} \\
\hline
\begin{array}{c} 0 \\ \vdots \\ 0  \end{array}&  B_n
\end{array}\right)}
\, \mbox{,} \,
y_n\definicio
{\scriptscriptstyle \left(\begin{array}{c|c}
1&  \begin{array}{ccc} 0&  \cdots&  0 \end{array} \\
\hline
\begin{array}{c} 0 \\ \vdots \\ 0  \end{array}&  A_n
\end{array}\right)} ,
$$
$$
z_n\definicio
{\scriptscriptstyle
\left(\begin{array}{c|c|c}
1&  \begin{array}{ccc} 0&   \cdots&  0 \end{array}&
\begin{array}{ccccc} 0&   \cdots&  0&  1&  0 \end{array}\\
\hline
\begin{array}{c} 0 \\ \vdots \\ 0  \end{array}&  A_{n-1}&  A_{n-1}C_{n-1}\\
\hline
\begin{array}{c} 0 \\ \vdots \\ 0  \end{array}&  0&  A_{n-1}
\end{array}\right)}
\, \mbox{ and }
$$
$$
t_n\definicio
{\scriptscriptstyle
\left(\begin{array}{c|c|c}
1&  \begin{array}{ccc} 0&   \cdots&  0 \end{array}&
\begin{array}{ccc} 0&   \cdots&  0 \end{array}\\
\hline
\begin{array}{c} 0 \\ \vdots \\ 0  \end{array}&  A_{n-1}&  A_{n-1}C_{n-1}\\
\hline
\begin{array}{c} 0 \\ \vdots \\ 0  \end{array}&  0&  A_{n-1}
\end{array}\right)} .
$$
\begin{lemma}\label{representations}
The $(2^{n}+1)\times(2^n+1)$ matrices $x_n$, $y_n$, $z_n$ and $t_n$ defined
above have the following properties:
\begin{enumerate}[\rm (a)]
\item $x_n$ has order $2^{n+1}$,
\item $y_n$ has order $2$,
\item $y_nx_ny_n=x_n^{-1}$,
\item $z_n^2=x_n^{2^n}$, in particular $z_n$ has order $4$,
\item $z_nx_nz_n^{-1}=x_n^{-1}$,
\item $t_n$ has order $2$ and
\item $t_nx_nt_n^{-1}=x_n^{2^n-1}$.
\end{enumerate}
So $x_n$ and $y_n$ generate a subgroup isomorphic to $D_{2^{n+2}}$; $x_n$
and $z_n$ a subgroup isomorphic to $Q_{2^{n+2}}$, and $x_n$ and $t_n$ a subgroup
isomorphic to $SD_{2^{n+2}}$.
\end{lemma}

\begin{proof}
\begin{enumerate}[(a)]
\item The order of $x_n$ has been computed in Lemma \ref{lemmaorderB}.
\item This is a direct computation using Lemma \ref{lemmaABC} (a).
\item This is equivalent to show that $(y_n x_n)^2=\Id$. We have
$$
(y_nx_n)^2=
{\scriptscriptstyle
\left(\begin{array}{c|c}
1&   \begin{array}{cccc} 1  & 0 & \cdots & 0\end{array}\\
\hline
 0 &  A_nB_n
\end{array}\right)^2} =
{\scriptscriptstyle
\left(\begin{array}{c|c}
1&  0 \\
\hline
 0 &  (A_nB_n)^2
\end{array}\right)} = \Id ,
$$
where we have used Lemma \ref{lemmaABC} (e) in the second equality to get zeros in the first row, 
and Lemma \ref{lemmaABC} (c) gives us the last step.
\item The computation of $x_n^{2^n}$ has been done in Lemma \ref{lemmaorderB}:
$$
x_n^{2^n}=
{\scriptscriptstyle
\left(\begin{array}{c|c}
1&  \begin{array}{cccc} 0&  \cdots &  0 &  1 \end{array} \\
\hline
 0 &  \Id
\end{array}\right)} ,
$$
A direct calculation gives us the expression:
$$
z_n^2=
{\scriptscriptstyle
\left(
\begin{array}{c|c|c}
1 & 0 & \begin{array}{cccc} 0 & \cdots & 0 & 1 \end{array} \\ \hline
0 & A_{n-1}^2 & A_{n-1}^2C_{n-1}+A_{n-1}C_{n-1}A_{n-1} \\ \hline
0 & 0 & A_{n-1}^2
\end{array}
\right) .
}
$$
Apply now Lemma \ref{lemmaABC} (a) to obtain the identity in the diagonal and Lemma \ref{lemmaABC} (e) to
get $A_{n-1}^2C_{n-1}+A_{n-1}C_{n-1}A_{n-1} = C_{n-1}+C_{n-1}=0$.
\item Using (b) and (c) we have to prove that $z_nx_nz_n^{-1}=y_n^{-1}x_ny_n$, which is the same as prove
$(y_n z_n) x_n = x_n (y_n z_n)$. Direct computations show us the following results:
$$
y_nz_n=\left(\begin{array}{c|c|c}
    1 & 0 & \begin{array}{ccccc} 0 & \cdots & 0 & 1 & 0 \end{array} \\ \hline
    0 & \Id & \Delta_{n-1} \\ \hline
    0 & 0 & \Id  \end{array}
\right)
$$
where we have used that $C_{n-1}+\Id=\Delta_{n-1}$.

Observe now that, as $C_{n-1}$ and $\Id$ commute with $B_{n-1}$, so does $\Delta_{n-1}$, and use it to check:
$$
(y_n z_n) x_n = 
\left(\begin{array}{c|c|c}
    1 &  \begin{array}{cccc} 1 & 0 & \cdots & 0 \end{array} & \begin{array}{ccccc} 0 & \cdots & 0 & 1 & 1 \end{array} \\ \hline
    0 & B_{n-1} & \Delta_{n-1}^T + \Delta_{n-1}B_{n-1} \\ \hline
    0 & 0 & B_{n-1}  \end{array}
\right) = x_n (y_n z_n) .
$$
\item The computation of $t_n^2$ is analogue to the computation of $z_n^2$ with the first row as the identity
matrix, getting $t_n^2=\Id$.
\item The statement is equivalent to show that $(t_nx_n)^2=x^{2^n}$. Multiplying directly the matrices
we have:
$$
t_n x_n = 
\left(\begin{array}{c|c|c}
    1 & \begin{array}{cccc} 1 & 0 & \cdots & 0  \end{array}& 0 \\ \hline
    0 & A_{n-1}B_{n-1} & A_{n-1}\Delta_{n-1}^T +A_{n-1}C_{n-1}B_{n-1} \\ \hline
    0 & 0 & B_{n-1}  \end{array} \right)  .
$$
Now use almost the same computations as in the proof of (c) to show that $(t_nx_n)^2$ has 
the same entries as
$x_n^{2^n}$: we use that the first row of 
$A_{n-1}\Delta_{n-1}^T$ is $(1 \, 0 \, \cdots \, 0)$, and that $A_{n-1}C_{n-1}B_{n-1}=
A_{n-1}B_{n-1}C_{n-1}$, so by Lemma \ref{lemmaABC} (e) and the multiplication 
by $C_{n-1}$, the first row of $A_{n-1}C_{n-1}B_{n-1}$ is $(1 \, 0 \, \cdots \, 0 \, 1)$. So
the sum of these two vectors give the same entries as the first row of $x_n^{2^n}$. 
\end{enumerate}

\end{proof}

Now we want to check that the previous representations of the dihedral,
quaternion and semidihedral groups are of minimal rank.
To check this we will compute the exponent of a Sylow 2-subgroup of a general
linear group over $\F_2$.

\begin{lemma}\label{minima_m}
Let $M$ be an element of order $2^{n}$ in $\GL_m(\F_2)$. Then
$2^{n-1}<m$. Moreover, there
exists an element of order $2^{n}$ in the linear group $\GL_{2^{n-1}+1}(\F_2)$.
\end{lemma}

\begin{proof}
Let $M$ be an element of order $2^n$. Then $M$ satisfies the polynomial
$X^{2^n}-1$. As we are in
characteristic $2$ we get $X^{2^n}-1=(X-1)^{2^n}$. The minimal
polynomial of $M$ must be of the
form $(X-1)^r$, and the characteristic polynomial must be $(X-1)^m$,
with $r \leq m$. If
$2^{n-1} \geq r$, then $M$ would satisfy the polynomial
$(X-1)^{2^{n-1}}=(X^{2^{n-1}}-1)$, so the order
of $M$ would be at most $2^{n-1}$. So we get $2^{n-1}<r\leq m$.

Finally, the element $x_{n-1}$ in Lemma \ref{representations} is of
order $2^n$ in $\GL_{2^{n-1}+1}(\F_2)$.
\end{proof}

\section{Projective resolutions and the Yoneda cocomplex}
\subsection{Quaternion groups}
We can find a projective resolution of $Q_{2^n}$ in \cite[pp.
253]{Cartan-Eilenberg}, but before giving it
we need some notation. Recall from Equation \eqref{notationG} the
notation for the generators and cohomology of the quaternion groups:
$$
Q_{2^n}=\langle x , z \mid x^{2^{n-1}}=1, z^2=x^{2^{n-2}} ,
zxz^{-1}=x^{-1}\rangle .
$$
Consider the following elements in the group algebra $\F_2[Q_{2^n}]$:
\begin{equation}\label{eq:QIJKLN}
\begin{array}{l}
I=1+x, \,  J=1+z,  \, K=1+xz, \, L=1+x+x^2+ \cdots + x^{2^{n-2}-1}, \\
N_x=1+x+x^2+ \cdots + x^{2^{n-1}-1} \text{ and } N=\sum_{g\in Q_{2^n}} g.
\end{array}
\end{equation}
\begin{lemma}\label{lemmarelationsQIJKLN}
Consider the elements $I$, $J$, $K$, $L$, $N_x$ and $N$ defined above. Then we have the 
following relations:
$$
\begin{array}{l}
L=I^{2^{n-2}-1}, \, N_x=I^{2^{n-1}-1}, \, I^{2^{n-2}}=J^2=K^2, \, I^{2^{n-1}}=J^4=K^4=0,\\
KI=IJ, \, K=I+J+IJ \text{ and } N=JN_x=N_xJ=KN_x=N_xK.
\end{array}
$$
\end{lemma}
\begin{proof} Recall that, as we are taking coefficients in $\F_2$, $(1+u)^{2^m}=1+u^{2^m}$
and $(1+u)^{2^m-1}=1+u+ \cdots + u^{2^m-1}$. So using this and the order of the elements $x$, $z$ and $xz$ in
$Q_{2^n}$ we get all the equalities in the first row. The first to equalities of the second row can be computed directly.
The last equality can be deduced from the fact that all elements in $Q_{2^n}$ can be written as
$x^iz^\varepsilon$, for $0 \leq i \leq 2^{n-1}-1$ and $0\leq \varepsilon \leq 1$ in just one way, so all elements
appear as a summand in the products $JN_x$, $N_xJ$, $KN_x$ and $N_xK$.
\end{proof}
The following result can be found in \cite[pp. 253]{Cartan-Eilenberg} and we include it here to formulate it
using the elements in Equation \eqref{eq:QIJKLN}:
\begin{lemma}\label{lemmaprojresQ}
A projective resolution of $\F_2$ as $\F_2[Q_{2^n}]$-module is given
by the following periodic data:
$$
\xymatrix{
\F_2&
P_0 \ar[l]_{\varepsilon}&
P_1 \ar[l]_{\partial_1}&
P_2 \ar[l]_{\partial_2}&
P_3 \ar[l]_{\partial_3}&
P_4 \ar[l]_{\partial_4}&
P_5 \ar[l]_{\partial_5} \cdots
}
$$
where $P_{4i}\cong P_{4i+3} \cong \F_2[Q_{2^n}]$ and $P_{4i+1}\cong
P_{4i+2} \cong \F_2[Q_{2^n}]^2(=\F_2[Q_{2^n}]\oplus\F_2[Q_{2^n}])$
and the differentials: $\partial_{4i+1}=\left( I \quad J\right)$,
$\partial_{4i+2}=\left(\begin{smallmatrix} L&  K \\ J&  I
\end{smallmatrix}\right)$,
$\partial_{4i+3}=\left(\begin{smallmatrix} I \\ K \end{smallmatrix}\right)$ and
$\partial_{4i}=(N)$.
\end{lemma}

Now we want to express the generators in the cohomology of quaternion groups as a cochain maps
in the Yoneda cocomplex. We refer the interested reader to \cite{Borge-Thesis} for the user of 
the Yoneda cocomplex in group cohomology.

Recall the cohomology of the Quaternion groups from Equation \eqref{notationHG}:
$$
\begin{array}{l}
H^*(BQ_8)\cong \F_2[X,Y,V]/(X^2+XY+Y^2,X^2Y+XY^2) \text{ and } \\
H^*(BQ_{2^n})\cong \F_2[X,Y,V]/(X^2+XY,Y^3)  \mbox{ for $n\geq 4$.}
\end{array}
$$
with $\deg(X)=\deg(Y)=1$ and $\deg(V)=4$.
\begin{lemma}\label{lemma:genHQ}
The generators $X$, $Y$ and $V$ in $H^*(BQ_{2^n})$ can be represented in the 
Yoneda cocomplex by the following cochain maps:
\begin{itemize}
\item The generator $X$ as the cochain map $X_i \colon P_{i+1} \to P_i$
defined as (to cover the case $Q_8$ here we are using the convention $I^0=1$):
$$
\begin{array}{l}
X_{4i}=(1\, 0),\,
X_{4i+1}=\left(\begin{smallmatrix} I^{2^{n-2}-2}&  1 \\ 0&  1+I
\end{smallmatrix}\right), \\
X_{4i+2}=\left(\begin{smallmatrix} 1 \\ 1 \end{smallmatrix}\right)
\text{ and }
X_{4i+3}=(I^{2^{n-1}-2}J).
\end{array}
$$
\item The element $Y$ in $H^1(BQ_{2^n})$
by a cochain map $Y_i \colon P_{i+1} \to P_i$
defined as:
$$
Y_{4i}=(0\quad 1), \,
Y_{4i+1}=\left(\begin{smallmatrix} 0&  1 \\ 1&  0\end{smallmatrix}\right), \,
Y_{4i+2}=\left(\begin{smallmatrix} 0 \\ 1 \end{smallmatrix}\right) \text{ and }
Y_{4i+3}=(N_x).
$$
\item Finally the element $V\in H^4(BQ_{2^n})$ by a cochain map defined by
$V_i \colon P_{i+4} \to P_i$ which is the identity.
\end{itemize}
\end{lemma}
\begin{proof}
We have that the projective resolution from Lemma \ref{lemmaprojresQ} is minimal, in the sense
that $P_i \cong \F_2[Q_{2^n}]^{\dim(H^i(Q_{2^n}))}$, so all $\F_2[Q_{2^n}]$-morphisms 
$E \colon P_1 \to \F_2$ must lift to maps $E_i$ to give the following commutative
diagram
$$
\xymatrix{
\F_2 &
P_0 \ar[l]_{\varepsilon} &
P_1 \ar[l]_{\partial_1} \ar@{.>}[d]_{E_0}  \ar[ld]_{E}& 
P_2 \ar[l]_{\partial_2}     \ar@{.>}[d]_{E_1} & 
P_3 \ar[l]_{\partial_3} \ar@{.>}[d]_{E_2} &
P_4 \ar[l]_{\partial_4} \ar@{.>}[d]_{E_3} & 
P_5 \ar[l]_{\partial_5} \ar@{.>}[d]_{E_4} \cdots \\
 & \F_2 & 
P_0 \ar[l]^{\varepsilon} & 
P_1 \ar[l]^{\partial_1} & 
P_2 \ar[l]^{\partial_2} & 
P_3 \ar[l]^{\partial_3} &
P_4 \ar[l]^{\partial_4}  \cdots 
}
$$
As $P_1=\F_2[Q_{2^n}]^2$ we can start looking for the lifting of
the $\F_2[Q_{2^n}]$-linear map $\tilde{X} \definicio (1 \quad 0)$ (respectively 
$\tilde{Y} \definicio (0 \quad 1)$). Observe that the map $X_j\colon P_{j+1} \to P_j$
(respectively $Y_j\colon P_{j+1} \to P_j$) defined in the statement is indeed 
a lifting for $\tilde X$ (respectively $\tilde Y$). This fact can be checked
multiplying the matrices $\partial_\bullet$ defined in Lemma \ref{lemmaprojresQ}
and $X_\bullet$ (respectively $Y_\bullet$) defined in the statement. To get the equality
of both sides one has to apply the relations given in Lemma \ref{lemmarelationsQIJKLN}. 

As elements in $H^*(Q_{2^n})$ we can compute the multiplication involving $\tilde X$ and
$\tilde Y$ composing the cochain maps. Doing that we will get that 
${\tilde X}^2+ \tilde X \tilde Y =0$ and ${\tilde Y}^3=0$, obtaining that $\tilde X$ (respectively $\tilde Y$)
can be took as 
a representative of $X$ in the Yoneda cocomplex (respectively $Y$).

Finally a non-zero element in $H^4(Q_{2^n})$ which will represent $V$ 
is a lifting of $(1) \colon P_4 \to \F_2$, and, as $P_{i+4} \cong P_{i}$ and the differential are also
periodic, we can take the identity.
\end{proof}

\subsection{Semidihedral groups}
In this subsection we will proceed in a similar way as the previous one. The main difference with
the quaternion family comes 
from the fact that, for semidihedral groups, there  there is not such a nice 
projective resolution.

Recall from Equation \eqref{notationG} a finite presentation for the semidihedral groups:
$$
SD_{2^n}=\langle x , t \mid x^{2^{n-1}}=1, t^2=1 ,
txt^{-1}=x^{2^{n-2}-1}\rangle .
$$
Consider the following elements in the group algebra $\F_2[SD_{2^n}]$:
\begin{equation} \label{eq:SDIJKLN}
\begin{array}{l}
I=1+x,\, J=1+t,\, L=1+x+x^2+\cdots+x^{2^{n-2}-2} \, \text{ and } \\
N_x=1+x+x^2+\cdots+x^{2^{n-1}-1}.
\end{array}
\end{equation}
The following equalities will be used in the computation of the generators in the Yoneda cocomplex
and their products.
\begin{lemma}\label{lemmarelationsSDIJKLN}
Consider the elements $I$, $J$, $L$ and $N_x$ defined above. Then we have the following relations:
\begin{enumerate}[\rm (a)]
\item $L=I^{2^{n-2}-1}+x^{2^{n-2}-1}$, $I^{2^{n-1}}=J^2=0$,
\item $(1+tL)I=IJ$, $(1+tL^n)I=I(1+tL^{n-1})$,
\item $tN_x=N_xt$, $tI^{2^{n-1}-2}=I^{2^{n-1}-2}t$,
\item $I^{2^{n-2}-2}t=tI^{2^{n-2}-2},$
\item $(1+tL^{2i})^2=0$, $(1+tL^{2i+1})^2=N_x$  and $L^{2^{n-1}}=1$.
\end{enumerate}
\end{lemma}
\begin{proof}
\begin{enumerate}[(a)]
\item One proceeds as the proof of the first part of Lemma \ref{lemmarelationsQIJKLN}.
\item This can be done directly. As example we do the first equality:
$IJ=(1+x)(1+t)=1+x+t+xt=1+x+t+tx^{2^{n-2}-1}$ and $(1+tL)I=(1+t+tx+\cdots tx^{2^{n-2}-2})(1+x)=1+x+t+tx^{2^{n-2}-1}$, where
in the last equality we are using that when we expand the product most of the summands cancel in $\F_2$.
\item In the first equality, we use the commutation rules between $t$ and $x$ in the presentation of $SD_{2^n}$. 
In the second the fact that $I^{2^{n-1}-2}$ is the sum of all possible even powers of $x$ (including $x^0=1$) and the 
commutation rules between $t$ and $x$ preserve the exponent of $x$ modulo $2$.
\item The element $I^{2^{n-2}-2}$ is the sum of all even powers of $x$ up to degree $2^{n-2}-2$. The commutation
rule between $x$ and $t$ adds $2^{n-2}$ to the exponent of $x$ and changes the sign, so at the end we have the same
exponents.
\item The first two equalities use the equality, in characteristic $2$, $(1+u)^2=1+u^2$ in each case, and applies the commutation
rules between $x$ and $t$ to get the desired result. The last statement follows also from the same equality in $\F_2$:
$L^{2^{n-1}}=1+x^{2^{n-1}}+(x^2)^{2^{n-1}}+ \cdots +(x^{2^{n-2}})^{2^{n-1}}$ and as $(x^i)^{2^{n-1}}=1$ we get the
sum of an odd number of $1$, so $L^{2^{n-1}}=1$.
\end{enumerate}

\end{proof}

We can find a projective resolution for these family of groups in \cite{Wall}:
\begin{lemma}\label{lemmaprojresSD}
A projective resolution of $\F_2$ as $\F_2[SD_{2^n}]$-module is given by:
$$
\xymatrix{
\F_2&
P_0 \ar[l]_{\varepsilon}&
P_1 \ar[l]_{\partial_1}&
P_2 \ar[l]_{\partial_2}&
P_3 \ar[l]_{\partial_3}&
P_4 \ar[l]_{\partial_4}&
P_5 \ar[l]_{\partial_5} \cdots
}
$$
where $P_i\cong \F_2[SD_{2^n}]^{i+1}(=\F_2[SD_{2^n}]\oplus\stackrel{i+1}{\dots}\oplus\F_2[SD_{2^n}])$
and the differentials defined inductively: $\partial_{1}=\left( I
\quad J\right)$,
$$
\partial_{2i} = \left( \begin{array}{c|ccccc}
 N_x&  1+tL^i&  0&  \cdots&  0 \\
\hline
0&  \multicolumn{3}{|c}{\multirow{3}{*}{$\partial_{2i-1}$}} \\
\vdots&  \\
0&
\end{array} \right) \quad \text{and}
$$
$$
\partial_{2i+1} = \left( \begin{array}{c|cccccc}
 I&  1+tL^i&  i&  0&  \cdots&  0 \\
\hline
0&  \multicolumn{4}{|c}{\multirow{3}{*}{$\partial_{2i}$}} \\
\vdots&  \\
0&
\end{array} \right) \quad ,
$$
where $I$, $J$, $L$ and $N_x$ are the ones defined in Equation \eqref{eq:SDIJKLN}.
\end{lemma}
Now we need the cohomology of $SD_{2^n}$, which has been introduced in Equation \eqref{notationHG}:
$$
H^*(BSD_{2^n}) \cong \F_2[X,Y,U,V]/(X^2\!+\!XY,XU,X^3,U^2+(X^2+Y^2)V) , 
$$
with $\deg(X)=\deg(Y)=1$, $\deg(U)=3$ and $\deg(V)=4$.

We proceed now giving cochain maps representing the generators. This free resolutions does
not allow us to give inductive formulas, but as we are just interested in $4$-fold iterated Massey
products of elements in degrees $1$ and $2$ we just need to give the first $4$ steps of the
maps.

Before giving the element we fix the following basis of this cohomology as a graded vector space in
low degrees: $\{X,Y\}$ in degree one, $\{X^2,Y^2\}$ in degree two, $\{U,Y^3\}$ in degree three and
$\{Y^4,YU,V\}$ in degree four.

\begin{lemma}\label{lemma:genHSD}
The representing cochain maps in the Yoneda cocomplex of the generators $X$, $Y$, $U$ and $V$ in $H^*(BSD_{2^n})$
are characterized by the following data:
\begin{itemize}
\item The element $X \in H^1(BSD_{2^n})$ can be taken
as the a cochain morphism $X_i\colon P_{i+1} \to P_{i}$ with
$$X_0=(1\quad 0), \quad X_1=\left( \begin{array}{c|cc} I^{2^{n-1}-2}&
t(L+I^{2^{n-2}-2})&  0 \\ \hline
                              0&  \multicolumn{2}{|c}{X_0}
                             \end{array}\right) ,
$$
$$
X_2=\left( \begin{array}{c|ccc} 1&  t(L+I^{2^{n-2}-2})&  1&  0 \\
          \hline 0&  \multicolumn{3}{|c}{\multirow{2}{*}{$X_1$}} \\
          0&
          \end{array}\right) \quad \text{and}
$$
$$ X_3=\left( \begin{array}{c|cccc} I^{2^{n-1}-2}&  0&  0&  0&
0 \\
          \hline 0&  \multicolumn{4}{|c}{\multirow{3}{*}{$X_2$}} \\
          0&   \\
          0&
          \end{array}\right) .
$$
\item The element $Y\in H^1(BSD_{2^n})$ can be taken as a cochain map $Y_i\colon P_{i+1} \to P_{i}$
with $Y_i=(0 | \Id_{i+1})$ where $\Id_{i+1}$ is the $(i+1)\times(i+1)$ identity matrix.
\item Either $U$ or $U+Y^3$ is represented by a lifting of the map from $P^3$ to $\F_2$
given by the matrix $(1\,0\,0\,0)$.
\item Either $V$ or $YU$ can be taken as a lifting of the map from $P^4$ to $\F_2$
given by the matrix $(1\,0\,0\,0\,0)$.
\end{itemize}
\end{lemma}
\begin{proof}
Here we give all the arguments which are used in the proof, but, for the sake of clarity, avoiding explicit computations. 
The computations we avoid can be checked directly using the relations in
Lemma \ref{lemmarelationsSDIJKLN}.

The elements in $H^1(BSD_{2^n})$ are determined by $\F_2[Q_{2^n}]$-morphisms $P_1 \to \F_2$. As
$P_1=\F_2[SD_{2^n}]^2=\dim(H^1(BSD_{2^n})$ both morphisms given by matrices $(1 \, 0)$ and $(0\, 1)$
must lift. The statement in the Lemma tells us which is one of the lifting for each one ($X_i$ and $Y_i$ respectively). 
Now we can decide which one is $X$ and $Y$ in $H^*(BSD_{2^n})$ checking, for example, the relation $X^3=0$.

The elements in $H^3(BSD_{2^n})$ are determined by  $\F_2[Q_{2^n}]$-morphisms $P_3 \to \F_2$
which can be lifted
to cochain maps, and we can write them as matrices $1 \times 4$. Using
the previous
representatives $Y^3$ is determined by $(0\,0\,0\,1)$, $X^3=X^2Y=XY^2$
(which is a coboundary) is
represented by $(0\,0\,1\,0)$. Observe that $(0\,1\,0\,0)$ cannot be lifted, so it does not represent any
element in $H^3(BSD_{2^n})$. So it remains $(1\,0\,0\,0)$, which
can be lifted and, taking into account the dimension of the cohomology in degree $3$,
it can be considered as a representative for $U$ or $U+Y^3$.

The same arguments work for detecting which can be $V$ as a $\F_2[Q_{2^n}]$-morphism from
$P_4 \to \F_2$:
$(0\,0\,0\,0\,1)$ determines $Y^4$, $(0\,0\,0\,1\,0)$ is a coboundary,
$(0\,0\,1\,0\,0)$ cannot be lifted,
$(0\,1\,0\,0\,0)$ determines $YU$, so $V$ (or $V+YU$) can be taken as $(1\,0\,0\,0\,0)$.
\end{proof}

\bibliographystyle{alpha}
\bibliography{biblio} 

\end{document}